\documentclass[12pt]{amsart}
\usepackage[margin=1.2in]{geometry}
\usepackage{graphicx,latexsym}
\usepackage{amsfonts, amssymb, amsmath, amsthm}

\newtheorem{theorem}{Theorem}[section]
\newtheorem{pro}[theorem]{Proposition}

\newtheorem{co}[theorem]{Corollary}

\theoremstyle{definition}
\newtheorem{de} [theorem]{Definition}

\theoremstyle{remark}
\newtheorem{rem}[theorem]{Remark}

 \numberwithin{equation}{section}

\begin{document}

\title[The ridgelet transform of distributions]{The ridgelet transform of distributions}

\author[S. Kostadinova]{Sanja Kostadinova}
\address{Faculty of Electrical Engineering and Information Technologies, Ss. Cyril and Methodius University, Rugjer Boshkovik bb, 1000 Skopje, Macedonia}
\email{ksanja@feit.ukim.edu.mk}

\author[S. Pilipovi\'{c}]{Stevan Pilipovi\'{c}}
\address{Department of Mathematics and Informatics, University of Novi Sad, Trg Dositeja Obradovi\'ca 4, 21000 Novi Sad, Serbia}
\email {stevan.pilipovic@dmi.uns.ac.rs}

\author[K. Saneva]{Katerina Saneva}
\address{Faculty of Electrical
Engineering and Information Technologies, Ss. Cyril
and Methodius University, Rugjer Boshkovik bb, 1000 Skopje, Macedonia}
\email{saneva@feit.ukim.edu.mk}

\author[J. Vindas]{Jasson Vindas}
\address{Department of Mathematics, Ghent University, Krijgslaan 281 Gebouw S22, 9000 Gent, Belgium}
\email{jvindas@cage.Ugent.be}


\begin{abstract}
We define and study the ridgelet transform of (Lizorkin) distributions. We establish connections with the Radon and wavelet transforms. 
\end{abstract}

\maketitle

\section{Introduction}

In  \cite{candes1,candes3} Cand\`{e}s introduced and studied the continuous ridgelet transform. He developed a harmonic analysis groundwork for this transform and showed that it is possible to obtain constructive and stable approximations of functions by ridgelets. Ridge functions often appear in the literature of approximation theory, statistics, and signal analysis.   One of the motivations for the introduction of the ``X-let'' transforms, such as the ridgelet or curvelet transforms, comes from the search of optimal representations of signals in high-dimensions. Wavelets are very good in detecting point singularities, but they have several difficulties in localizing edges of higher dimension \cite{candes2}. The ridgelet transform is more sensitive to higher dimensional discontinuities, as it essentially projects a hyperplane singularity into a point singularity and then takes a one-dimensional wavelet transform.

In this paper we provide a thorough analysis of the ridgelet transform and its transpose, called here the ridgelet synthesis operator, on various test function spaces. Our main results are continuity theorems on such function spaces (cf. Section \ref{ridgelet test functions}). We then use our results to develop a distributional framework for the ridgelet transform. Distribution theory is a power tool in applied mathematics and the extension of integral transforms to generalized function spaces is an important subject with a long tradition. 
It should be noticed that Roopkumar has proposed a different definition for the ridgelet transform of distributions \cite{RT,roopkumar}; however, his work contains several major errors (see Remark \ref{remark Roopkumar} below).

Let $\mathbb{Y}^{n+1}=\mathbb{S}^{n-1}\times\mathbb{R}\times\mathbb{R}_{+}$, where $\mathbb{S}^{n-1}$ is the unit sphere of $\mathbb{R}^{n}$. Cand\`{e}s showed \cite{candes3} that the ridgelet transform $\mathcal{R}_{\psi}:L^{2}(\mathbb{R}^{n})\to L^{2}(\mathbb{Y}^{n+1})$ is a constant multiple of an isometry, provided that $\psi$ is admissible (cf. Section \ref{extended formulas}). We will show in Section \ref{ridgelet distributions} that the ridgelet transform and the ridgelet synthesis operator can be extended as continuous mappings
$\mathcal{R}_{\psi}:\mathcal{S}'_{0}(\mathbb{\mathbb{R}}^{n})\to\mathcal{S}'(\mathbb{Y}^{n+1})$ and $\mathcal{R}_{\psi}^{t}:\mathcal{S}'(\mathbb{Y}^{n+1})\to \mathcal{S}'_{0}(\mathbb{\mathbb{R}}^{n})$. Here $\mathcal{S}'(\mathbb{Y}^{n+1})$ is a certain space of distributions of slow growth on $\mathbb{Y}^{n+1}$ and $\mathcal{S}'_{0}(\mathbb{R}^{n})$ stands for the Lizorkin distribution space (cf. Subsection \ref{spaces}). We remark that the Lizorkin spaces play a key role in Holschneider's approach to the wavelet transform of distributions \cite{hols}. Many important Schwartz distribution spaces, such as $\mathcal{E}'(\mathbb{R}^{n})$, $\mathcal{O}'_{C}(\mathbb{R}^{n})$, $L^{p}(\mathbb{R}^{n})$, or the $\mathcal{D}'_{L^{p}}(\mathbb{R}^{n})$ spaces, are embedded into $\mathcal{S}'_{0}(\mathbb{R}^{n})$. 

The ridgelet transform of distributions must be more carefully handled than the wavelet transform. While the wavelet transform of a distribution can be defined by direct evaluation of the distribution at the wavelets, this procedure fails for the ridgelet transform because the ridgelets do not belong to the Schwartz class $\mathcal{S}(\mathbb{R}^{n})$. The larger distribution space where the direct approach works is $\mathcal{D}'_{L^{1}}(\mathbb{R}^{n})$ (cf. \ref{ridgelet transform functions}). We treat the ridgelet transform on $\mathcal{S}'_{0}(\mathbb{R}^{n})$ via a duality approach. The crucial continuity results for test function spaces are given in Section \ref{ridgelet test functions}.

The ridgelet transform is intimately connected with the Radon and wavelet transforms. Helgason \cite{helgason} proved range theorems for the Radon and dual Radon transforms on the Lizorkin test function spaces $\mathcal{S}_{0}$. In Section \ref{section Radon transform} we apply our continuity theorems for the ridgelet transform to discuss the continuity of the Radon transform on these spaces and their duals. The Radon transform on Lizorkin spaces naturally extends the one considered by Hertle \cite{hertle2} on various distribution spaces.  We use in Section \ref{section desingularization} ideas from the theory of tensor products of topological vector spaces to study the relation between the distributional ridgelet, Radon, and wavelet transforms. We conclude this article with a desingularization formula, which essentially shows that the ridgelet transform of a Lizorkin distribution is smooth in the position and scale variables.

We point out that the wavelet transform has shown usefulness to study pointwise scaling properties of distributions \cite{hols,meyer,PV,SV,sohn,v-p-r}. One can then expect that the ridgelet transform of distributions might provide a tool for studying higher dimensional scaling notions, such as those introduced by \L ojasiewicz in \cite{lojasiewicz2}.
\section{Preliminaries}

\subsection{Spaces}\label{spaces}
In this subsection we explain the spaces to be employed in this paper. We use the notation $\mathbb H=\mathbb R\times \mathbb R_{+}$, so that ${\mathbb Y^{n+1}}=\mathbb{S}^{n-1}\times\mathbb{H}=\{(\mathbf{u},b,a):\: \mathbf{u} \in {\mathbb S^{n-1}}, b
\in {{\mathbb R}}, a\in{\mathbb R}_{+} \}$ ,
where we recall that ${\mathbb S^{n-1}}$ stands for the unit sphere of $\mathbb R^{n}$. We always assume that $n\geq2$. We use the constants in the Fourier transform as
$$
\widehat{\phi}(\mathbf{w})
=\int_{{\mathbb R}^n} \phi (\mathbf {x})e^{-i\mathbf{x}\cdot
\mathbf{w}}d{\mathbf x}.
$$

We provide all distribution spaces with the strong dual topologies. Besides ${\mathcal S}({{\mathbb
R}}^n)$ and ${\mathcal S}'({{\mathbb
R}}^n)$, we employ the already mentioned Schwartz space $\mathcal{D}'_{L^{1}}(\mathbb{R}^{n})$, defined in Schwartz' book \cite[p. 200]{scwartz}. The space $\mathcal{D}'_{L^{1}}(\mathbb{R}^{n})$ contains the space of compactly supported distributions $\mathcal{E}'(\mathbb{R}^{n})$ and the space of convolutors $\mathcal{O}_{C}'(\mathbb{R}^{n})$. 
Of crucial importance for our study is the Lizorkin test function space ${{\mathcal S}}_0({{\mathbb
R}}^n)$ of highly time-frequency localized functions over ${{\mathbb
R}}^n$ \cite{hols}. It consists of those elements of ${\mathcal S}({{\mathbb R}}^n)$ having
all moments equal to $0$, namely, $\phi\in {{\mathcal S}}_0({\mathbb{R}}^{n})$ if
$$
\int_{{{\mathbb R}}^n}{\mathbf x^m\phi(\mathbf x)d\mathbf x}=0, \ \ \ \mbox{for all } m\in {{\mathbb N}}_0^n.
$$
It is a closed subspace of  ${\mathcal S}({{\mathbb R}}^n)$.
Let us point out that other authors use a different notion for this space. For instance, Helgason \cite{helgason} denotes $\mathcal S_{0}(\mathbb R^{n})$ by $\mathcal S^{*}(\mathbb R^{n}).$ 
Its dual space $\mathcal{S}'_{0}(\mathbb{R}^{n})$, known as the space of Lizorkin distributions,
 is canonically isomorphic to the quotient of $\mathcal{S}'(\mathbb{R}^{n})$ by the space of polynomials; the quotient projection $\mathcal{S}'(\mathbb{R}^{n})\to\mathcal{S}'_{0}(\mathbb{R}^{n})$ is explicitly given by the restriction of tempered distributions to $\mathcal{S}_{0}(\mathbb{R}^{n})$. This quotient projection is injective on $\mathcal{D}'_{L^{1}}(\mathbb{R}^{n})$; therefore, we can regard $\mathcal{D}_{L^{1}}'(\mathbb{R}^{n})$, $\mathcal{E}'(\mathbb{R}^{n})$, and $\mathcal{O}'_{C}(\mathbb{R}^{n})$ as (dense) subspaces of $\mathcal{S}_{0}'(\mathbb{R}^{n})$.

We denote by $\mathcal{D}(\mathbb{S}^{n-1})$ the space of smooth functions on the sphere. Given a locally convex space $\mathcal{A}$ of smooth test functions on $\mathbb{R}$, we write $\mathcal{A}(\mathbb{S}^{n-1}\times\mathbb{R})$ for the space of functions $\varrho(\mathbf{u},p)$ having the properties of $\mathcal{A}$ in the variable $p\in\mathbb{R}$ and being smooth in $\mathbf{u}\in\mathbb{S}^{n-1}$. 
 

We introduce ${\mathcal S}({\mathbb Y^{n+1}})$ as the space of functions ${\Phi }\in C^{\infty
}({\mathbb Y^{n+1}})$ satisfying the decay conditions
\begin{equation} \label{eqNorms}
\rho_{s,r}^{l,m,k}(\Phi)={\mathop{\sup }_{(\mathbf{u},b,a)\in {\mathbb Y^{n+1}}}
\left(a^{s}+\frac{1}{a^{s}}\right){{{\rm (1+}b^{2}{\rm )}}^{r/2}}
\left|\frac{{\partial
}^l}{\partial a^l}\frac{{\partial }^m}{\partial
b^m}{\triangle_{\mathbf{u}}^{k}}{ \Phi
}\left(\mathbf{u},b,a \right)\right|{\rm \ }<\infty {\rm \ }\ }
\end{equation}
for all $l,m,k,s,r \in {{\mathbb N}}_0$, where ${\triangle_{\mathbf{u}}}$ is the Laplace-Beltrami operator on the unit sphere $\mathbb S^{n-1}$. The topology of this space is defined by means of the seminorms (\ref{eqNorms}). Its dual ${\mathcal S}'({\mathbb Y^{n+1}})$ will be fundamental in our definition of the ridgelet transform of Lizorkin distributions, as it contains the range of this transform (cf. Section \ref{ridgelet distributions}).
We follow the ensuing convention. We fix $a^{-n}{d\mathbf{u}}dbda$ as the \emph{standard measure} on $\mathbb{Y}^{n+1}$. Here $d\mathbf u$ stands for the surface measure on the sphere $\mathbb{S}^{n-1}$. 
Accordingly, our convention for identifying a locally integrable function $F$ on $\mathbb{Y}^{n+1}$ with a distribution on $\mathbb{Y}^{n+1}$ is as follows. If it is of slow growth on $\mathbb{Y}^{n+1}$, namely, it satisfies the bound
\[\left|F\left(\mathbf{u},b,a \right)\right|\leq C(1+\left|b\right|)^{s}\left(a^{s}+\frac{1}{a^s}\right), \ \ \ \left(\mathbf{u},b,a
\right) \in \mathbb {Y}^{n+1},\]
for some $s,C>0$, we shall always identify $F$ with an element of ${\mathcal S}'({\mathbb Y^{n+1}})$ via
\begin{equation}
\label{regular1}\left\langle F,\Phi \right\rangle :={\int^{\infty }_{0 }{\int^{\infty }_{-\infty}\int_{\mathbb S^{n-1}}
{F\left(\mathbf{u},b,a \right)\Phi \left(\mathbf{u},b,a
\right)\frac{d\mathbf{u}dbda}{a^{n}}}}}, \ \  \ \Phi \in \ {\mathcal
S}\left({\mathbb Y^{n+1}}\right).
\end{equation}

A related space is $\mathcal{S}(\mathbb{H})$, the space of highly localized test functions on the upper half-plane \cite{hols}. Its elements are smooth functions $\Psi$ on $\mathbb{H}$ that satisfy
$$\mathop{\sup }_{(b,a)\in {\mathbb H}}\left(a^{s}+\frac{1}{a^s}\right)
(1+b^{2})^{r/2}
\left|\frac{{\partial }^m}{\partial
b^m}\frac{{\partial
}^l}{\partial a^l}{ \Psi
}\left(b,a \right)\right|{\rm \ }<\infty,
$$
for all $l,m,s,r \in {{\mathbb N}}_0$; its topology being defined in the canonical way \cite{hols}.

Observe that the nuclearity of the Schwartz spaces \cite{treves} immediately yields the equalities $\mathcal{S}(\mathbb{Y}^{n+1})=\mathcal{D}(\mathbb{S}^{n-1})\hat{\otimes}\mathcal{S}(\mathbb{H})$, $\mathcal{S}(\mathbb{S}^{n-1}\times \mathbb{R})=\mathcal{D}(\mathbb{S}^{n-1})\hat{\otimes}\mathcal{S}(\mathbb{R})$, and $\mathcal{S}_{0}(\mathbb{S}^{n-1}\times \mathbb{R})=\mathcal{D}(\mathbb{S}^{n-1})\hat{\otimes}\mathcal{S}_{0}(\mathbb{R})$, where $X\hat{\otimes}Y$ is the topological tensor product space obtained as the completion of $X\otimes Y$ in, say, the $\pi$-topology  or the $\varepsilon-$topology \cite{treves}.

\subsection{The ridgelet transform of functions and some distributions}\label{ridgelet transform functions}
Let $\psi \in {\mathcal S}({\mathbb R})$. For $\left(\mathbf{u},b,a \right)\in {\mathbb Y^{n+1}}$, where $\mathbf{u}$ is the orientation parameter, $b$ is the location parameter, and $a$ is the scale parameter, we define the  function ${\psi
}_{\mathbf{u},b,a }:{{\mathbb R}}^n\to {\mathbb C}$, called
\emph{ridgelet}, as
\[{\psi }_{\mathbf{u},b,a }\left(\mathbf{x}\right)=
\frac{1}{a}\psi
\left(\frac{\mathbf x\cdot \mathbf u -b}{a}\right),\  \  \ {\mathbf
x}\in {{\mathbb R}}^n.\]
This function is constant along hyperplanes $\mathbf x\cdot \mathbf u
= \textnormal{const.}$, called ``ridges".  In the orthogonal direction it is a
wavelet, hence the name ridgelet. The function $\psi$ is often referred in the literature \cite{candes1,candes3} as a \emph{neuronal activation function}. The ridgelet transform ${\mathcal R}_{\psi}f$ of an integrable function $ f\in
L^1({{\mathbb R}}^n)$ is defined by
\begin{equation}\label{ridgelet}
{{\mathcal R}}_{\psi}f\left(\mathbf{u},b,a \right)=\int_{\mathbb
R^n}{f(\mathbf{x}){\overline{\psi
}_{\mathbf{u},b,a}}(\mathbf{x})d\mathbf{x}}=\left\langle f(\mathbf x),{\overline{\psi
}_{\mathbf{u},b,a}}(\mathbf{x})\right\rangle_\mathbf{x}.\end{equation}
\noindent  where $\left(\mathbf{u},b,a \right)\in {\mathbb
Y^{n+1}}$.
 
The ridgelet transform can also be canonically defined for distributions $f\in \mathcal{D}'_{L^{1}}(\mathbb{R}^{n})$ via (\ref{ridgelet}), because the test function ${\psi
}_{\mathbf{u},b,a }\in \mathcal{D}_{L^{\infty}}(\mathbb{R}^{n})$ and thus the integral formula can be still interpreted in the sense of Schwartz integrable distributions \cite[p. 203]{scwartz}. In particular, (\ref{ridgelet}) makes sense for $f\in\mathcal{E}'(\mathbb{R}^{n})$ or $f\in\mathcal{O}'_{C}(\mathbb{R}^{n})$. On the other hand,
if one wishes to extend the definition of the ridgelet transform to more general spaces than $\mathcal{D}'_{L^{1}}(\mathbb{R}^{n})$, one must proceed with care. Even in the $L^{2}$ case, (\ref{ridgelet}) is not directly extendable to $f\in L^{2}(\mathbb{R}^{n})$ because the defining integral might fail to converge. A similar difficulty is faced when trying to extend the ridgelet transform to distributions: the function $\psi_{\mathbf{u},b,a}\notin \mathcal{S}(\mathbb{R}^{n})$ and therefore (\ref{ridgelet}) is not well defined for $f\in\mathcal{S}'(\mathbb{R}^{n})$. We shall overcome this difficulty in Section \ref{ridgelet distributions} via a duality approach and define the ridgelet transform of Lizorkin distributions for $\psi\in\mathcal{S}_{0}(\mathbb{R})$.

\subsection{The continuous wavelet transform} Given functions $f$ and $\psi$, the wavelet transform $\mathcal W_{\psi}f(b,a)$ of $f$ is defined by
\begin{equation}\label{eqwavelett}\mathcal W_{\psi}f(b,a)=\int_{\mathbb R}f(x) \frac{1}{a}\overline{\psi}\Big(\frac{x-b}{a}\Big)dx, \ \ \ (b,a)\in\mathbb{H}.
\end{equation}
The expression (\ref{eqwavelett}) is defined, e.g., if $f,\psi\in L^{2}(\mathbb{R})$, $f\in L^{1}(\mathbb{R})$ and $\psi\in L^{\infty}(\mathbb{R})$, or in other circumstances. We will actually work with the wavelet transform of distributions. So if $f\in\mathcal{S}'(\mathbb{R})$ and $\psi \in \mathcal S(\mathbb R)$ (or $f\in\mathcal{S}'_{0}(\mathbb{R})$ and $\psi \in \mathcal S_{0}(\mathbb R)$), one replaces (\ref{eqwavelett}) by
\begin{equation}
\label{wavelett2}\mathcal W_{\psi}f(b,a)= \left\langle  f(x), \frac{1}{a}\overline{\psi}\Big(\frac{x-b}{a}\Big)\right\rangle_{x}, \ \ \ (b,a)\in\mathbb{H}.
\end{equation}
We refer to Holschneider's book \cite{hols} for a distribution wavelet transform theory based on the spaces $\mathcal{S}_{0}(\mathbb{R})$, $\mathcal{S}(\mathbb{H})$, $\mathcal{S}'_{0}(\mathbb{R})$, and $\mathcal{S}'(\mathbb{H})$. For the wavelet transform of vector-valued distributions, we refer to \cite[Sect. 5 and 8]{PV}.

\subsection{The Radon transform}
Let $f$ be a function that is integrable on hyperplanes of $\mathbb R^{n}$. For $\mathbf u\in\mathbb{S}^{n-1}$ and $p\in
\mathbb{R}$, the equation $\mathbf{x}\cdot
\mathbf{u}=p$ specifies a hyperplane of $\mathbb R^n$. Then, the Radon transform of $f$ is defined as
$$ Rf(\mathbf{u},p)=Rf_{\mathbf{u}}(p):=\int_{\mathbf{x}\cdot\mathbf{u}=p}{f(\mathbf{x})d\mathbf{x}}=\int_{\mathbb R^n}{f(\mathbf{x})\delta(p-\mathbf{x}
\cdot \mathbf{u})d\mathbf{x}},$$
where $\delta$ is the Dirac delta. Fubini's theorem ensures that if $f\in L^{1}(\mathbb{R}^{n})$, then   $Rf\in L^{1}(\mathbb{S}^{n-1}\times \mathbb{R})$.
The Fourier transform and the Radon transform are connected by the
so-called \emph{Fourier slice
theorem} \cite{helgason};
according to it, the
Radon transform can be computed as
\begin{equation}\label{Fourierslice}Rf(\mathbf{u},p)=\frac{1}{2\pi}\int_{-\infty}^{\infty}{\widehat{f}(\omega\mathbf{u})e^{ i p\omega } d\omega}, \  \  \  \mathbf u\in\mathbb{S}^{n-1}, \: p\in\mathbb {R},
\end{equation}
for sufficiently regular $f$ (e.g., for $f\in L^{1}(\mathbb{R}^{n})$ such that $\widehat{f}\in L^{1}(\mathbb{R}^{n})$).

The dual Radon transform (or back-projection) $R^{\ast}\varrho$ of the function $\varrho \in L^{\infty}(\mathbb{S}^{n-1}\times \mathbb{R})$ is defined as
\begin{equation*}
R^{\ast}\varrho(\mathbf{x})=\int_{\mathbb S^{n-1}}\varrho(\mathbf{u},\mathbf{x}\cdot \mathbf{u})d\mathbf{u}.
\end{equation*}
The transforms $R$ and $R^{\ast}$ are then formal transposes, i.e., 
\begin{equation}\label{Radonduality}
\left\langle R f,\varrho \right\rangle=\left\langle f,R^{\ast}\varrho \right\rangle.
\end{equation} For instance, for $f\in L^{1}(\mathbb{R}^{n})$ and $\varrho\in L^{\infty}(\mathbb{S}^{n-1}\times \mathbb{R})$,
$$
\int_{\mathbb{R}^{n}} f(\mathbf{x})R^{\ast}\varrho(\mathbf{x})d\mathbf{x}=\int_{-\infty}^{\infty}\int_{\mathbb{S}^{n-1}}Rf(\mathbf{u},p)\varrho(\mathbf{u},p)d\mathbf{u} dp.
$$
More details on the Radon transform can be found in Helgason's book \cite{helgason}.  See also  \cite{gelfand5,hertle2,hertle1,ludwig,ram}. In particularly, Hertle \cite{hertle2} has exploited the duality relation (\ref{Radonduality}) to extend the definition of the Radon transform as a continuous map between  various distribution spaces.  In fact, the dual Radon transform $R^{\ast}:\mathcal{A}(\mathbb{S}^{n-1}\times\mathbb{R})\to \mathcal{A}(\mathbb{R}^{n})$ is continuous for $\mathcal{A}=\mathcal{D}_{L^{1}},\mathcal{E},\mathcal{O}_{C}$ and the Radon transform can then be defined on their duals by transposition as in (\ref{Radonduality}). In Section \ref{section Radon transform} we will enlarge the domain of the Radon transform to the Lizorkin distribution space $\mathcal{S}'_{0}(\mathbb{R}^{n})$. 

\subsection{Relation between the Radon, ridgelet and wavelet transforms}

The ridgelet transform is intimately connected with the
Radon transform. Changing variables in (\ref{ridgelet}) to $\mathbf{x}=p\mathbf{u}+\mathbf{y}$, where $p\in\mathbb{R}$ and $\mathbf{y}$ runs over the hyperplane perpendicular to $\mathbf{u}$, one readily obtains
\begin{equation}\label{rad-rid}{{\mathcal R}}_{\psi}f\left(\mathbf{u},b,a \right)
=\mathcal W_{\psi}(Rf_{\mathbf{u}})(b,a),
\end{equation}
where $\mathcal{W}_{\psi}$ is a one-dimensional wavelet transform. The relation (\ref{rad-rid}) holds if $f\in L^{1}(\mathbb{R}^{n})$. (In fact, we will extend its range of validity in Sections \ref{ridgelet distributions} and \ref{section desingularization}.)
Thus, ridgelet analysis can be seen as a form of wavelet analysis
in the Radon domain, i.e., the ridgelet transform is precisely the
application of a one-dimensional wavelet transform to the slices
of the Radon transform where $\mathbf{u}$ remains fixed and $p$
varies.
Furthermore, 
by the Fourier slice
theorem (\ref{Fourierslice}) and the relation (\ref{rad-rid}), we get the useful formula
\begin{equation}\label{GrindEQ__2_6_}
{\mathcal R}_{\psi}f\left(\mathbf{u},b,a \right)
= \frac{1}{2\pi}\int^{\infty }_{-\infty
}{\widehat{f}\left(\omega\mathbf{u}\right)\overline{\widehat{\psi }}(a\omega
)e^{ib\omega }d\omega }.
\end{equation}

\section{Extended reconstruction formulas and Parseval relations}\label{extended formulas}
In \cite{candes3} (see also \cite[Chap. 2]{candes1}), Cand\`es has established reproducing formulas and Parseval's identities for the ridgelet transform under the assumption that $\psi\in\mathcal{S}(\mathbb{R})$ is an \emph{admissible neuronal activation function}, meaning that it satisfies the constrain
\begin{equation}
\label{admissible}
\int_{-\infty}^{\infty} \frac{|\widehat{\psi}(\omega)|^{2}}{\left|\omega\right|^{n}}d\omega<\infty.
\end{equation}
We shall establish in this section more general reconstruction and Parseval's formulas employing neuronal activation functions which are not necessarily admissible.
The crucial notion involved in our analysis is given in the next definition. As usual, a function is called non-trivial if it is not the zero function.

\begin{de}\label{nondegenerate} Let $\psi\in\mathcal{S}(\mathbb{R})$ be a non-trivial test function. A test function $\eta\in\mathcal{S}(\mathbb{R})$ is said to be a \emph{reconstruction neuronal activation function} for $\psi$ if the constant
\begin{equation}\label{admiss}
K_{\psi, \eta}:=(2\pi)^{n-1}\int^{\infty }_{-\infty}\overline{\widehat{\psi }}(\omega){\widehat{\eta }}(\omega)\frac{d\omega}{|\omega|^{n}}
\end{equation}
is non-zero and finite.\end{de}

\smallskip

 It is then easy to show that \emph{any} $\psi$  admits a reconstruction neuronal activation function $\eta$, as long as $\psi$ is non-trivial, and, in such a case, one may take $\eta\in\mathcal{S}_{0}(\mathbb{R})$, if needed. Our first result states that it is always possible to do ridgelet reconstruction for non-trivial neuronal activation functions.

\begin{pro}[Reconstruction formula]\label{theoremreconstruction} Let $\psi\in\mathcal{S}(\mathbb{R})$ be non-trivial and let $\eta\in\mathcal{S}(\mathbb{R})$ be a reconstruction neuronal activation function for it. If $f\in L^{1}(\mathbb{R}^{n})$ is such that  $\widehat{f}\in L^{1}(\mathbb{R}^{n})$, then the following reconstruction formula holds pointwisely,

\begin{equation} \label{reconstruction1}
f\left(\mathbf{x}\right)=
\frac{1}{K_{\psi,\eta}}
\int_{\mathbb{S}^{n-1}}\int^{\infty
}_{0}\int^{\infty }_{-\infty }
{\mathcal R}_{\psi}f\left(\mathbf{u},b,a
\right){\eta }_{\mathbf{u},b,a}(
\mathbf{x})\frac{dbdad\mathbf{u}}{a^{n}}.
\end{equation}
\end{pro}

\begin{rem}
Proposition \ref{theoremreconstruction} shows that ridgelet reconstruction is possible for non-oscillatory neuronal activation functions; indeed, for test functions that might not satisfy the admissibility condition (\ref{admissible}) (e.g., the Gaussian $\psi(x)=e^{-x^{2}}$). Nevertheless, if $\psi$ is not oscillatory, then the reconstruction function $\eta$ should compensate this fact by having its first $n+1$ moments equal to $0$.
\end{rem}

\begin{proof}
Indeed, (\ref{GrindEQ__2_6_}) yields
\begin{align*}
&\int_{\mathbb S^{n-1}}{\int^{\infty }_{0}{\int^{\infty
}_{-\infty }{{{\mathcal R}}_{\psi}f\left(\mathbf{u},b,a \right){\eta
}_{\mathbf{u},b,a }({ \mathbf{x}})\frac{dbda{d\mathbf{u} }}{a^{n}}}}}
\\
&
=\frac{1}{2\pi}\int^{\infty }_{-\infty }\int_{\mathbb S^{n-1}}\int^{\infty }_{0}
e^{ i \omega\mathbf{u}\cdot\mathbf{x}}
\overline{\widehat{\psi}}(\omega a)\widehat{\eta}(\omega a)\widehat{f}(\omega\mathbf{u})\frac{dad\mathbf{u}
d\omega}{a^{n}}
\\
&=
\frac{K_{\psi,\eta}}{(2\pi)^{n}}\int^{\infty}_{0}\int_{\mathbb{S}^{n-1}}e^{i \omega\mathbf{u}\cdot\mathbf{x}}
\omega^{n-1}\widehat{f}(\omega\mathbf{u})d\mathbf{u}
d\omega.
\end{align*}
\end{proof}

A similar calculation leads to the ensuing result.

\begin{pro}[Extended Parseval's relation]\label{proposition Parseval} Let $\psi\in\mathcal{S}(\mathbb{R})$ be non-trivial and let $\eta\in\mathcal{S}(\mathbb{R})$ be a reconstruction neuronal activation function for it. Then,
\begin{equation}\label{parseval}
\int_{\mathbb{R}^{n}}f(\mathbf{x})g(\mathbf{x}) d\mathbf{x}=\frac{1}{K_{\psi,\eta}}\int_{\mathbb{S}^{n-1}}\int^{\infty }_{0
}\int^{\infty }_{-\infty} {{\mathcal R}_{\psi}{f}(\mathbf{u},b,a){\mathcal
R}_{\overline{\eta}}{g}(\mathbf{u},b,a)}\frac{dbdad\mathbf{u}}{a^n},
\end{equation}
for any $f,g\in L^{1}(\mathbb{R}^{n})\cap L^{2}(\mathbb{R}^{n})$.
\end{pro}

According to our choice of the standard measure on $\mathbb{Y}^{n+1}$ (cf. Subsection \ref{spaces}), we denote by $L^{2}(\mathbb{Y}^{n+1}):=L^{2}(\mathbb{Y}^{n+1}, a^{-n}{d\mathbf{u} }dbda)$ so that the inner product on this space is
$$
(F,G)_{L^{2}(\mathbb{Y}^{n+1})}:={\int^{\infty }_{0 }{\int^{\infty }_{-\infty}
\int_{\mathbb S^{n-1}}{F\left(\mathbf{u},b,a \right)\overline{G} \left(\mathbf{u},b,a
\right)\frac{{d\mathbf{u}}dbda}{a^{n}}}}}.
$$
As already observed by Cand\`{e}s \cite{candes3}, the transform $\sqrt{K^{-1}_{\psi,\psi}}\mathcal{R}_{\psi}$ is $L^{2}$-norm preserving whenever $\psi$ is an admissible function. In such a case
$||\mathcal{R}_{\psi}||_{L^{2}(\mathbb{Y}^{n+1})}=K_{\psi,\psi}||f||_{L^{2}(\mathbb{R}^{n})}$ on a dense subspace of $L^{2}(\mathbb{R}^{n})$, as follows from  (\ref{parseval}). Consequently, $\mathcal{R}_{\psi}$ extends to a constant multiple of an isometric embedding $L^{2}(\mathbb{R}^{n})\to L^{2}(\mathbb{Y}^{n+1})$.

The reconstruction formula (\ref{reconstruction1}) suggests to define an operator that maps functions on $\mathbb{Y}^{n+1}$ to functions on $\mathbb{R}^{n}$ as superposition of ridgelets. Given $\psi\in\mathcal{S}(\mathbb{R}^{n})$, we introduce the \emph{ridgelet synthesis operator} as  

\begin{equation}
\label{synthesis}
\mathcal{R}_{\psi}^{t} \Phi(\mathbf{x}):= \int_{\mathbb{S}^{n-1}}\int^{\infty
}_{0}\int^{\infty }_{-\infty } \Phi(\mathbf{u},b,a){\psi }_{\mathbf{u},b,a}(
\mathbf{x})\frac{dbdad\mathbf{u}}{a^{n}}, \  \  \  \mathbf{x}\in\mathbb{R}^{n}.
\end{equation}
The integral (\ref{synthesis}) is absolutely convergent, for instance, if $\Phi \in \mathcal{S}(\mathbb{Y}^{n+1})$. In Section \ref{ridgelet test functions} we will show that if $\psi\in\mathcal{S}_{0}(\mathbb{R})$, then $\mathcal{R}^{t}_{\psi}$ maps continuously $\mathcal{S}(\mathbb{Y}^{n+1})\to\mathcal{S}_{0}(\mathbb{R}^{n})$. It will then be shown in Section \ref{ridgelet distributions} that  $\mathcal{R}^{t}_{\psi}$ can be even extended to act on the distribution space $\mathcal{S}'(\mathbb{Y}^{n+1})$.
Observe that the relation (\ref{reconstruction1}) takes the form $(\mathcal{R}^{t}_{\eta}\circ\mathcal{R}_{\psi})f=K_{\psi,\eta} f$.

We remark that $\mathcal{R}^{t}_{\overline{\psi}}$  and $\mathcal{R}_{\psi}$ are actually formal transposes. The proof of the next proposition is left to the reader, it is a simple consequence of Fubini's theorem.

\begin{pro}\label{transpose proposition} Let $\psi\in\mathcal{S}(\mathbb{R})$. If $f\in L^{1}(\mathbb{R}^{n})$ and $\Phi\in \mathcal{S}(\mathbb{Y}^{n+1})$, then
\begin{equation}
\label{transpose1}
\int_{\mathbb{R}^{n}}f(\mathbf{x}) \mathcal{R}^{t}_{\psi} \Phi(\mathbf{x})d\mathbf{x}=\int_{\mathbb{S}^{n-1}}\int^{\infty
}_{0}\int^{\infty }_{-\infty } \mathcal{R}_{\overline{\psi}}f(\mathbf{u},b,a)\Phi(\mathbf{u},b,a) \frac{dbdad\mathbf{u}}{a^n}.
\end{equation}
\end{pro} 
Following our convention for regular distributions on $\mathbb{Y}^{n+1}$ (cf. (\ref{regular1})), we may write (\ref{transpose1}) as
\begin{equation*}
\left\langle f,\mathcal{R}_{\bar{\psi}}^{t}\Phi\right\rangle=\left\langle\mathcal{R}_{\psi} f,\Phi\right\rangle.
\end{equation*}
The above dual relation will be the model for our definition of the distributional ridgelet transform.

\section{ Continuity of the ridgelet transform on test function spaces} \label{ridgelet test functions}
The aim of the section is to prove that the
ridgelet mappings 
$${{\mathcal R}}_{\psi}:{{\mathcal S}}_0({{\mathbb R}}^n)\to
{\mathcal S}({\mathbb Y^{n+1}})\  \  \  \  \mbox{ and } \  \  \  \ {\mathcal R}^{t}_{\psi}: {\mathcal
S}({\mathbb Y^{n+1}})\to {\mathcal S}_{0}({{\mathbb R}}^n)$$ are
continuous when $\psi\in\mathcal{S}_{0}(\mathbb{R})$. For non-trivial $\psi$, the ridgelet transform $\mathcal{R}_{\psi}$ is injective and $\mathcal{R}_{\psi}^{t}$ is surjective, due to the reconstruction formula (cf. Proposition \ref{theoremreconstruction}). Recall that we endow ${\mathcal S}({\mathbb Y^{n+1}})$ with the system of seminorms (\ref{eqNorms}). 

Notice that we can extend the definition of the ridgelet transform as a sesquilinear mapping 
$$
\mathcal R: (f, \psi)\mapsto \mathcal R_{\psi}f,
$$
whereas the ridgelet synthesis operator extends to the bilinear form 
$$
\mathcal R^{t}: (\Phi, \psi)\mapsto \mathcal {R}^{t}_{\psi}\Phi.
$$


\begin{theorem}\label{continuity theorem ridgelet}The ridgelet mapping ${{\mathcal R}}:{{\mathcal S}}_0({{\mathbb R}}^n)\times \mathcal S_{0}(\mathbb R)\to {\mathcal S}({\mathbb Y^{n+1}})\ $ is continuous.
\end{theorem}

\begin{proof} For the seminorms on $\mathcal{S}_{0}(\mathbb R^{n})$, we make the choice

\begin{equation}\label{eqNormS}\rho_{\nu}(\phi)=\sup_{\mathbf{x}\in \mathbb R^{n}, |m|\leq \nu}(1+|\mathbf{x}|)^{\nu}
\left|
\phi^{(m)}(\mathbf{x})\right|, \ \ \ \nu\in\mathbb{N}_{0}.
\end{equation}
We will show that, given $s,r,m,l,k \in {{\mathbb N}}_0$, there exist $\nu,\tau \in \mathbb N$ and $C>0$ such that

\begin{equation}\label{GrindEQ__3_1_}
\rho_{s,r}^{l,m,k}(\mathcal R_{\psi}\phi)\leq C \rho_{\nu}(\phi)\rho_{\tau}(\psi), \ \ \  \phi\in\mathcal{S}_{0}(\mathbb{R}^{n}),\: \psi\in\mathcal{S}_{0}(\mathbb{R}).
\end{equation}
We may assume that $r$ is even and $s\geq 1$. We divide the proof into six steps.

1. Using the definition of the ridgelet transform and the Leibniz formula, we have
\begin{align*}
&\left|\frac{\partial^l}{\partial a^l}\frac{\partial^m}{\partial b^m}{{\mathcal
R}}_{\psi}\phi\left(\mathbf{u},b,a \right)\right|
\\
&
=
\sum^l_{j=0}\frac{C_{m,l,j}}{a^{m+l-j}}\underset{d\leq 2j}{\sum_{i,q\leq j}}a^{-d-1}B_{m,l,j}^{i,q,d}
\left|\int_{\mathbb R^{n}} \phi\left({\mathbf x}\right)
\psi^{(m+i)} \left(\frac{\mathbf{x}\cdot\mathbf{u} -b}{a}\right)
(\mathbf{x}\cdot\mathbf{u}-b)^qd\mathbf{x} \right|
\\
&
\leq C \left(a^{m+2l}+\frac{1}{a^{m+2l}}\right)({1+b^{2}})^{l/2}\sum_{|\alpha|,i\leq l}\left|\frac{1}{a}\int_{\mathbb R^{n}} \mathbf{x^{\alpha}} \phi\left({\mathbf x}\right)
\psi^{(m+i)}\left(\frac{\mathbf{x}\cdot\mathbf{u} -b}{a}\right)
d\mathbf{x} \right|.
\end{align*}
Setting $\phi_{\alpha}(\mathbf{x})=\mathbf{x}^{\alpha}\phi(\mathbf{x})$, this yields
$$
\rho_{s,r}^{l,m,k}(\mathcal R_{\psi}\phi)\leq C 
\underset{|\alpha|\leq l}{\sum_{j\leq m+l}}\rho_{s+m+2l,r+l}^{0,0,k}(\mathcal{R}_{\psi^{(j)}}(\phi_{\alpha})). 
$$
So we can assume that $m=l=0$ because multiplication by $\mathbf{x}^{\alpha}$ and differentiation are continuous operators on  $\mathcal S_0.$

2. We now show that we may assume that $k=0$. Notice that

\begin{align*}\triangle_{\mathbf{u}}^{k}{{\mathcal
R}}_{\psi}\phi\left(\mathbf{u},b,a\right)&=\triangle_{\mathbf{u}}^{k}\int_{\mathbb R^{n}} \phi\left({\mathbf x}\right)a^{-1}\psi \left(\frac{\mathbf{x}\cdot\mathbf{u} -b}{a}\right)d\mathbf{x}
\\
&
={\sum_{|\alpha|,j,d\leq 2k}}a^{-d}P_{\alpha,j,d}(\mathbf{u})\frac{1}{a}\int_{\mathbb R^{n}} \mathbf{x^{\alpha}} \phi\left({\mathbf x}\right)
\psi^{(j)}\left(\frac{\mathbf{x}\cdot\mathbf{u} -b}{a}\right)
d\mathbf{x},
\end{align*}
where the $P_{\alpha,j,d}(\mathbf{u})$ are certain polynomials. The $P_{\alpha,j,d}$ are bounded, thus
$$
\left|\triangle_{\mathbf{u}}^{k}{{\mathcal
R}}_{\psi}\phi\left(\mathbf{u},b,a\right)\right|\leq C \left(a^{2k}+\frac{1}{a^{2k}}\right)\sum_{|\alpha|,j\leq 2k}\left|\frac{1}{a}\int_{\mathbb R^{n}} \mathbf{x^{\alpha}} \phi\left({\mathbf x}\right)
\psi^{(j)}\left(\frac{\mathbf{x}\cdot\mathbf{u} -b}{a}\right)
d\mathbf{x}\right|.
$$
This gives (with $\phi_{\alpha}$ as before)
$$
\rho_{s,r}^{0,0,k}(\mathcal R_{\psi}\phi)\leq C 
\sum_{|\alpha|,j\leq 2k}\rho_{2k+s,r}^{0,0,0}(\mathcal R_{\psi^{(j)}}
(\phi_{\alpha})). 
$$
Reasoning as above,  we can assume that $k=0$.

3. Observe that, by (\ref{GrindEQ__2_6_}),
\begin{align*}
\left(1+b^{2}\right)^{r/2}{{\mathcal
R}}_{\psi}\phi\left(\mathbf{u},b,a \right)&=\frac{1}{2\pi}\int_{-\infty}^{\infty} \widehat{\phi}(\omega\mathbf{u})\overline{\widehat{\psi}}(a\omega)\left(1-\frac{\partial^2}{\partial \omega^{2}}\right)^{r/2}e^{ib\omega}d\omega
\\
&
=\frac{1}{2\pi}\int_{-\infty}^{\infty}e^{ib\omega}\left(1-\frac{\partial^2}{\partial \omega^{2}}\right)^{r/2}(\widehat{\phi}(\omega\mathbf{u})\overline{\widehat{\psi}}(a\omega))d\omega
\\
&
=
\sum_{|\alpha|,j\leq r}a^{j}Q_{\alpha,j}(\mathbf{u})\int_{-\infty}^{\infty}e^{ib\omega} \widehat{\phi}^{(\alpha)}(\omega\mathbf{u})\overline{\widehat{\psi}^{(j)}}(a\omega) d\omega,
\end{align*}
for some polynomials $Q_{\alpha,j}$. Taking (\ref{GrindEQ__2_6_}) into account, and writing $\psi_{j}(x)=x^{j}\psi(x)$ and again $\phi_{\alpha}(\mathbf{x})=\mathbf{x}^{\alpha}\phi(\mathbf{x})$, we conclude that
$$
\rho^{0,0,0}_{s,r}(\mathcal{R}_{\psi}\phi)\leq C \sum_{|\alpha|,j\leq r}\rho_{s+r,0}^{0,0,0}(\mathcal R_{\psi_{j}}
(\phi_{\alpha})). 
$$
Consequently, we can assume $r=0$.

4.  We consider the part involving multiplication by $a^s$ in $\rho^{0,0,0}_{s,0}$.
Using the Taylor expansion of $\widehat{\phi}$, we obtain
\begin{align*}
a^s\left|{\mathcal R}_{\psi}\phi\left({\mathbf{u}},b,a \right)\right|&=\frac{a^{s}}{2\pi}
\left|\int_{-\infty}^{\infty} \widehat{\phi}(\omega{\mathbf{u}})\overline{\widehat{\psi}}(a\omega)e^{ib\omega}d\omega \right|
\\
&
\leq \sum_{|\alpha|= s-1} \frac{a^{s}}{2\pi}\left|\int_{-\infty}^{\infty}
\frac{(\omega{\mathbf{u}})^{\alpha}}{\alpha!}\widehat{\phi}^{(\alpha)}
(\omega_0{\mathbf{u}})
\overline{\widehat{\psi}}(a\omega)e^{ib\omega}
d\omega\right|
\\
&
\leq
\left(\sum_{|\alpha|= s-1}\frac{1}{2\pi \alpha!} \int_{\mathbb{R}^{n}} |\mathbf{x}^{\alpha}\phi(\mathbf{x})|d\mathbf{x}\right) \int_{-\infty}^{\infty} \left|\omega^{s-1}\widehat{\psi}(\omega)\right|
d\omega
\\
&
\leq C \rho_{s+n}(\phi) \rho_{s+1}(\psi).
\end{align*}

5. For the multiplication by $a^{-s},$ we develop $\widehat{\psi}$ into its Taylor expansion of order $s$. Then,
\begin{align*}a^{-s}\left|{{\mathcal R}}_{\psi}\phi\left({\mathbf{u}},b,a \right)\right|&=
\frac{1}{2\pi a^{s}}\left|\int_{-\infty}^{\infty} \widehat{\phi}(\omega{\mathbf{u}})\overline{\widehat{\psi}}(a\omega)e^{ib\omega}d\omega\right|
\\
&
\leq \frac{1}{2\pi s!}\int_{-\infty}^{\infty}
 |\omega^s\widehat{\phi}(\omega{\mathbf{u}})\overline{\widehat{\psi}^{(s)}}
(a\omega_0)|
d\omega.
\end{align*}
It is easy to see that last integral is less than $C \rho_{s+n+1}(\phi)\rho_{s+2}(\psi).$ Combining this fact with the bound from step 4, we obtain
$$
\rho^{0,0,0}_{s,0}(\mathcal{R}_{\psi}\phi)\leq C \rho_{s+n+1}(\phi)\rho_{s+2}(\psi). 
$$ 

6. Summing up all the estimates, we find that (\ref{GrindEQ__3_1_}) holds with $\nu=s+2r+4l+4k+m+n+1$ and $\tau=s+2r+4l+4k+2m+2$. This completes the proof.
\end{proof}

We now study the ridgelet synthesis operator.

\begin{theorem} \label{continuity theorem inverse} The bilinear mapping ${{{\mathcal R}}^{t}}:
{\mathcal S}({\mathbb Y^{n+1}})\times \mathcal{S}_{0}(\mathbb{R})\to {{\mathcal S}}_{0}({{\mathbb R}}^n)$ is continuous. \end{theorem}

\begin{proof} Let us first verify that the ridgelet synthesis operator has the claimed range, that is, we show that if  $\psi\in\mathcal{S}_{0}(\mathbb{R})$ and $\Phi\in\mathcal{S}(\mathbb{Y}^{n+1})$, then  $\phi(\mathbf{x}):=\mathcal{R}^{t}_{\psi}\Phi\in \mathcal{S}_{0}(\mathbb{R}^{n})$. In other words, we have to prove that 
\begin{equation}\label{eq4.19}
\lim_{\mathbf{w}\to0}\frac{\widehat{\phi}(\mathbf{w})}{|\mathbf{w}|^{k}}=0, \  \  \  \forall k\in\mathbb{N}_{0}.
\end{equation}
Observe that
$$
\phi(\mathbf{x})= \frac{1}{2\pi}\int_{\mathbb{S}^{n-1}}\int_{-\infty}^{\infty} \omega^{n-1}e^{i \omega \mathbf{u}\cdot\mathbf{x}}\left(\int_{0}^{\infty}\widehat{\Phi}(\mathbf{u},\omega,a)\frac{\widehat{\psi}(\omega a)}{(\omega a)^{n-1}}\frac{da}{a}\right)d \omega d\mathbf{u};
$$
hence, by Fourier inversion in polar coordinates,
\begin{equation}
\label{rid-Fourier-polar}
\widehat{\phi}(\omega\mathbf{u})=(2\pi)^{n-1}\int_{0}^{\infty}\left(\widehat{\Phi}(\mathbf{u},\omega,a)\frac{\widehat{\psi}(\omega a)}{(\omega a)^{n-1}}+\widehat{\Phi}(\mathbf{u},-\omega,a)\frac{\widehat{\psi}(-\omega a)}{(-\omega a)^{n-1}}\right )\frac{da}{a},
\end{equation}
$\omega\in\mathbb{R}_{+},\: \mathbf{u}\in\mathbb{S}^{n-1}.$ (Here $\widehat{\Phi}$ stands for the Fourier transform of $\Phi(\mathbf{u},b,a)$ with respect to the variable $b$.) Since $\Phi$ belongs to $\mathcal{S}(\mathbb{Y}^{n+1})$,  we have that for any $k\in\mathbb{N}$ we can find a constant $C_{k}>0$ such that $|\widehat{\Phi}(\mathbf{u},\omega,a)|\leq C_{k}a^{-k-1}$, uniformly for $\omega\in\mathbb{R}$ and $\mathbf{u}\in\mathbb{S}^{n-1}$. Thus, 
$$
\left|\widehat{\phi}(\omega\mathbf{u})\right|\leq C_{k} \int_{-\infty}^{\infty}\frac{|\widehat{\psi}(\omega a)|}{|\omega a |^{n-1}}\frac{da}{|a|^{k+2}}= C_{k} \omega^{k+1}\int_{-\infty}^{\infty}\left|\frac{\widehat{\psi}( a)}{a ^{n+k+1}}\right|da, \  \  \  \omega\in\mathbb{R},\: \mathbf{u}\in\mathbb{S}^{n-1},
$$
whence (\ref{eq4.19}) follows.

We now prove the continuity of the bilinear ridgelet synthesis mapping. Since the Fourier transforms $\psi\mapsto\widehat{\psi}$ and $\Phi\mapsto\widehat{\Phi}$ are continuous automorphisms on the $\mathcal{S}$ spaces, the families (cf. (\ref{eqNormS}) and (\ref{eqNorms}))
$$
\hat{\rho}_{\nu}(\psi)= \rho_{\nu}(\widehat{\psi}), \  \  \  \psi\in\mathcal{S}_{0}(\mathbb{R}), \  \  \  \nu=0,1,\dots, 
$$
and
$$
\hat{\rho}^{l,m,k}_{s,r}(\Phi)= \rho^{l,m,k}_{s,r}(\widehat{\Phi}), \  \  \ \Phi\in\mathcal{S}(\mathbb{Y}^{n+1}), \  \  \ \ l,m,k,s,r\in\mathbb{N}_{0},
$$
are bases of seminorms for the topologies of $\mathcal{S}_{0}(\mathbb{R})$ and $\mathcal{S}(\mathbb{Y}^{n+1})$, respectively.  We shall need a different family of seminorms on $\mathcal{S}_{0}(\mathbb{R}^{n})$. Observe first that the Fourier transform provides a Fr\'{e}chet space isomorphism from  $\mathcal{S}_{0}(\mathbb{R}^{n})$ onto $\mathcal{S}_{\ast}(\mathbb{R}^{n})$, the closed subspace of $\mathcal{S}(\mathbb{R}^{n})$ consisting of all those test functions that vanish at the origin together with all their partial derivatives. On the other hand, polar coordinates $\varphi(\omega \mathbf{u})$ provide a continuous mapping $\mathcal{S}_{\ast}(\mathbb{R}^{n})\to\mathcal{S}_{\ast}(\mathbb{S}^{n-1}\times\mathbb{R})$; the range of this mapping is closed (it consists of even test functions, i.e., $\varrho(-\mathbf{u},-\omega)=\varrho(\mathbf{u},\omega)$ \cite{drozhzhinov-z4,helgason}), and therefore the open mapping theorem implies that it is an isomorphism into its image.
Summarizing, the seminorms $\dot{\rho}_{N,q,k}$, given by
$$
\dot{\rho}_{N,q,k}(\phi):=\sup_{(\mathbf{u},\omega)\in\mathbb{S}^{n-1}\times\mathbb{R}} \left|\omega^{N} \frac{\partial^{q}}{\partial \omega^{q}}\Delta^{k}_{\mathbf{u}}\widehat{\phi}(\omega \mathbf{u})\right|, \   \   \   N,q,k\in\mathbb{N}_{0},
$$
are a base of continuous seminorms for the topology of $\mathcal{S}_{0}(\mathbb{R}^{n})$. We show that given $N,q,k\in\mathbb{N}_{0}$ there are $C>0$ and $\nu\in\mathbb{N}$ such that  
$$
\dot{\rho}_{N,q,k}\left(\mathcal{R}^{t}_{\psi}\Phi\right)\leq C \hat{\rho}_{n-1+q}(\psi)\sum_{m,s\leq \nu}\hat{\rho}^{0,m,k}_{s,N}(\Phi). 
$$
Now, setting again $\phi(\mathbf{x}):=\mathcal{R}^{t}_{\psi}\Phi\in \mathcal{S}_{0}(\mathbb{R}^{n})$, using the expression (\ref{rid-Fourier-polar}), the Leibniz formula, and the Taylor expansion for $\psi$, we get 
\begin{align*}
\left| \omega^{N}\frac{\partial^{q}}{\partial \omega^{q}}\Delta^{k}_{\mathbf{u}}\widehat{\phi}(\omega \mathbf{u})\right|
&
\leq  C \sum_{j=0}^{q}\sum_{d=0}^{j} \int_{-\infty}^{\infty}\left| a^{-j-1}\omega^{N}\frac{\partial^{q-j}}{\partial\omega^{q-j}}\Delta^{k}_{\mathbf{u}} \widehat{\Phi}(\mathbf{u},\omega,a)\frac{\widehat{\psi}^{(j-d)}(\omega a)}{(\omega a)^{n-1+d}}\right|da
\\
&
=C \sum_{j=0}^{q}\sum_{d=0}^{j} \int_{-\infty}^{\infty}\left| a^{-j-1}\omega^{N}\frac{\partial^{q-j}}{\partial\omega^{q-j}}\Delta^{k}_{\mathbf{u}}\widehat{\Phi}(\mathbf{u},\omega,a)\widehat{\psi}^{(j+n-1)}(\omega_{0} a)\right|da
\\
& 
\leq 
C\hat{\rho}_{n-1+q}(\psi) \sum_{j=0}^{q}(j+1)\hat{\rho}^{0,q-j,k}_{j+3,N}(\Phi)\int_{-\infty}^{\infty}\frac{a^{2}da}{a^{4}+1},
\end{align*}
as claimed.
\end{proof}

For future use, it is convenient to introduce wavelet analysis on $\mathcal{S}(\mathbb{S}^{n-1}\times \mathbb{R})$. Given $\psi\in\mathcal{S}(\mathbb{R})$, we let $\mathcal{W}_{\psi}$ act on the real variable $p$ of functions $g(\mathbf{u},p)$ (or distributions), that is,
\begin{equation}\label{waveletS}
\mathcal{W}_{\psi} g(\mathbf{u},b,a):= \int_{-\infty}^{\infty} \frac{1}{a}\overline{\psi}\left(\frac{p-b}{a}\right) g(\mathbf{u},p)dp= \left\langle  g(\mathbf{u},p), \frac{1}{a}\overline{\psi}\Big(\frac{p-b}{a}\Big)\right\rangle_{p},
\end{equation}
$(\mathbf{u},b,a)\in\mathbb{Y}^{n+1}$. Similarly, we define the wavelet synthesis operator on $\mathcal{S}(\mathbb{Y}^{n+1})$ as
\begin{equation}\label{synthesisS}
\mathcal{M}_{\psi}\Phi(\mathbf{u},p)= \int_{0}^{\infty}\int_{-\infty}^{\infty} \frac{1}{a}\psi\left(\frac{p-b}{a}\right)\Phi(\mathbf{u},b,a)\frac{dbda}{a}.
\end{equation}
A straightforward variant of the method employed in the proofs of Theorem \ref{continuity theorem ridgelet} and Theorem \ref{continuity theorem inverse} applies to show the following continuity result. Alternatively, since $\mathcal{S}(\mathbb{S}^{n-1}\times \mathbb{R})=\mathcal{D}(\mathbb{S}^{n-1})\hat{\otimes}\mathcal{S}(\mathbb{R})$, $\mathcal{S}_{0}(\mathbb{S}^{n-1}\times \mathbb{R})=\mathcal{D}(\mathbb{S}^{n-1})\hat{\otimes}\mathcal{S}_{0}(\mathbb{R})$ and $\mathcal{S}(\mathbb{Y}^{n+1})=\mathcal{D}(\mathbb{S}^{n-1})\hat{\otimes}\mathcal{S}(\mathbb{H})$, the result may also be deduced from a tensor product argument and the continuity of the corresponding mappings on $\mathcal{S}(\mathbb{R})$, $\mathcal{S}_{0}(\mathbb{R})$, and $\mathcal{S}(\mathbb{H})$ (cf. \cite{hols} or \cite{prv}).

\begin{co}\label{waveletanalysisS} The mappings
\begin{itemize}
\item [$(i)$] $\mathcal{W}:\mathcal{S}_0
( \mathbb{S}^{n-1}\times\mathbb{R})\times \mathcal{S}_0
(\mathbb{R})\rightarrow \mathcal{S} (\mathbb{Y}^{n + 1})$
\item [$(ii)$] $ \mathcal{M}: \mathcal{S}
(\mathbb{Y}^{n+1})\times \mathcal{S} (\mathbb{R})\rightarrow \mathcal{S}
(\mathbb{S}^{n-1}\times\mathbb{R}) $
\item [$(iii)$] $\mathcal{M}: \mathcal{S}
(\mathbb{Y}^{n+1})\times \mathcal{S}_{0} (\mathbb{R})\rightarrow \mathcal{S}_{0}
(\mathbb{S}^{n-1}\times\mathbb{R}) $
\end{itemize}
are continuous.
\end{co}
We end this section with a remark concerning reference \cite{RT}.

\begin{rem}\label{remark Roopkumar} In dimension $n=2$, Roopkumar has considered \cite{RT} the analogs of our Theorem \ref{continuity theorem ridgelet} and Theorem \ref{continuity theorem inverse} for the space $\mathcal{S}_{\#}(\mathbb{Y}^{n+1})$, where  $\mathcal{S}_{\#}(\mathbb{Y}^{n+1})$ consists of all those smooth functions $\Phi$ on $\mathbb{Y}^{n+1}$ satisfying
$$
\gamma^{l,m,k}_{s,r}(\Phi):= \sup_{(\mathbf{u},b,a)\in \mathbb {Y}^{n+1}}
\left|a^{s}b^{r}\frac{{\partial
}^l}{\partial a^l}\frac{{\partial }^m}{\partial
b^m}{\triangle_{\mathbf{u}}^{k}}{ \Phi
}\left(\mathbf{u},b,a \right)\right|{\rm \ }<\infty, \ \  \  \ \  l,m,k,s,r\in\mathbb{N}_{0}.
$$
Observe that his system of seminorms $\{\gamma^{l,m,k}_{s,r}\}$ does not take decay into account for small values of the scaling variable $a$ (the term $a^{-s}$ does not occur in his considerations). He claims \cite[Thrm. 3.1 and 3.3]{RT} to have shown that $\mathcal{R}_{\psi}:\mathcal{S}(\mathbb{R}^{2})\to \mathcal{S}_{\#}(\mathbb{Y}^{3})$ and $\mathcal{R}^{t}_{\psi}:\mathcal{S}_{\#}(\mathbb{Y}^{3})\to \mathcal{S}(\mathbb{R}^{2})$ are continuous when $\psi\in\mathcal{S}(\mathbb{R})$ satisfies the admissibility condition (\ref{admissible}). His proof of the continuity of 
$\mathcal{R}^{t}_{\psi}:\mathcal{S}_{\#}(\mathbb{Y}^{3})\to \mathcal{S}(\mathbb{R}^{2})$ appears to be incorrect because it seems to make use of the erroneous relation $x_{1}\cos\theta+x_{2}\sin\theta= (x_1+ix_2)e^{i\theta}$ \cite[p. 436]{RT}. Furthermore, his result on the continuity of $\mathcal{R}_{\psi}:\mathcal{S}(\mathbb{R}^{2})\to \mathcal{S}_{\#}(\mathbb{Y}^{3})$ turns out to be false because the ridgelet transform $\mathcal{R}_{\psi}$ does not even map $\mathcal{S}(\mathbb{R}^{n})$ into Roopkumar's space $\mathcal{S}_{\#}(\mathbb{Y}^{n+1})$. We show the latter fact with the following example. Choose the admissible function $\widehat{\psi}(\omega)=2\pi^{-n/2+1}\omega^{2n}e^{-\omega^{2}/4}$, $\omega\in\mathbb{R}$, and $\phi(\mathbf{w})=e^{-|\mathbf{w}|^{2}}$, $\mathbf{w}\in\mathbb{R}^{n}$. Then, by (\ref{GrindEQ__2_6_}),
\begin{align*}
\mathcal{R}_{\psi}\phi (\mathbf{u},0,a)&= \int_{-\infty}^{\infty}e^{-\omega^{2}/4}(a\omega)^{2n}e^{-(a\omega)^{2}/4}d\omega=\frac{1}{a}\int_{-\infty}^{\infty}e^{-\omega^{2}/(4a^{2})}\omega^{2n}e^{-\omega^{2}/4}d\omega
\\
&\sim \frac{1}{a}\int_{-\infty}^{\infty}\omega^{2n}e^{-\omega^{2}/4}d\omega=\frac{c}{a}, \ \  \ a\to\infty, 
\end{align*}
where $c\neq 0$. This shows that $\gamma^{0,0,0}_{2,0}(\mathcal{R}_{\psi}\phi)=\infty$. Therefore, $\mathcal{R}_{\psi}\phi\notin\mathcal{S}_{\#}(\mathbb{Y}^{n+1})$.
\end{rem}


\section{The ridgelet transform on $\mathcal{S}'_{0}({\mathbb{R}^{n}})$}
\label{ridgelet distributions}

We are ready to define the ridgelet transform of Lizorkin distributions. 

\begin{de} \label{def1} Let $\psi\in\mathcal S_{0}(\mathbb R)$. We define the ridgelet transform of $f\in\mathcal S'_{0}(\mathbb R^n)$ with respect to $\psi$ as the element $\mathcal{R}_{\psi}f\in \mathcal{S}'(\mathbb{Y}^{n+1})$ whose action on test functions is given by 
\begin{equation}\label{def11}\langle \mathcal{R}_{\psi}{f}, \Phi\rangle:=\langle f, \mathcal{R}^{t}_{\overline{\psi}}{\Phi}\rangle, \ \ \  \Phi \in \mathcal{S}(\mathbb{Y}^{n+1}).\end{equation}
\end{de}

The consistence of Definition \ref{def1} is guaranteed by Theorem \ref{continuity theorem inverse}. Likewise, Theorem \ref{continuity theorem ridgelet} allows us  to define the ridgelet synthesis operator $\mathcal{R}_{\psi}^{t}$ for $\psi\in\mathcal{S}_{0}(\mathbb{R})$ as a linear mapping from $\mathcal{S}'(\mathbb{Y}^{n+1})$ to $\mathcal{S}'_{0}({\mathbb{R}^{n}})$ (and not to $\mathcal{S}'({\mathbb{R}^{n}})$).

\begin{de}\label{def3} Let $\psi\in\mathcal{S}_{0}(\mathbb{R})$. The ridgelet synthesis operator $\mathcal{R}^{t}_{\psi} :\mathcal{S}'(\mathbb{Y}^{n+1})\to\mathcal{S}'_{0}({\mathbb{R}^{n}})$ is defined as
\begin{equation}\label{def21}\langle \mathcal{R}^{t}_{\psi}F, \phi\rangle:=\langle F, \mathcal{R}_{\overline{\psi}}{\phi}\rangle, \  \  \ F\in\mathcal{S}' (\mathbb{Y}^{n+1}),\ \phi \in \mathcal S({\mathbb R^{n}}).\end{equation}
\end{de}

Taking transposes in Theorems \ref{continuity theorem ridgelet} and  \ref{continuity theorem inverse}, we immediately obtain the ensuing continuity result.

\begin{pro} 
\label{pr1} Let $\psi\in\mathcal{S}_{0}(\mathbb{R})$. The ridgelet transform $\mathcal{R}_{\psi}:\mathcal{S}'_{0}({\mathbb{R}^{n}})\to \mathcal{S}'(\mathbb{Y}^{n+1})$ and the ridgelet synthesis operator $\mathcal{R}^{t}_{\psi} :\mathcal{S}'(\mathbb{Y}^{n+1})\to\mathcal{S}'_{0}({\mathbb{R}^{n}})$ are  continuous linear maps.
\end{pro}

We can generalize the reconstruction formula (\ref{reconstruction1}) to distributions.

\begin{theorem}[Inversion formula] \label{reconstruction distributions} Let $\psi\in\mathcal{S}_{0}(\mathbb{R})$ be non-trivial. If $\eta\in\mathcal{S}_{0}(\mathbb{R})$ is a reconstruction neuronal activation function for $\psi$, then
\begin{equation}\label{eqidentity}
\operatorname{id}_{\mathcal{S}_{0}'(\mathbb {R}^n)}=\frac{1}{K_{\psi,\eta}}(\mathcal{R}_{\eta}^{t}\circ \mathcal{R_{\psi}}).
\end{equation}
\end{theorem}

\begin{proof} Applying Definition \ref{def1}, Definition \ref{def3}, and Proposition \ref{theoremreconstruction}, we obtain at once
$$
\langle \mathcal{R}_{\eta}^{t}( \mathcal{R_{\psi}}f),\phi\rangle= \langle f,\mathcal{R}_{\overline{\psi}}^{t}( \mathcal{R_{\overline{\eta}}}\:\phi)\rangle=K_{\overline{\eta},\overline{\psi}}\langle f,\phi\rangle={K_{\psi,\eta}}\langle f,\phi\rangle.
$$
\end{proof}

In Subsection \ref{ridgelet transform functions} we have given a different definition of the ridgelet transform of distributions $f\in\mathcal{D}' _{L^{1}}(\mathbb{R}^{n})$ via the formula (\ref{ridgelet}). We now show that Definition \ref{def1} is consistent with (\ref{ridgelet}) (under our convention (\ref{regular1}) for identifying functions with distributions on $\mathbb{Y}^{n+1}$). In particular, our definition of the ridgelet transform for distributions is consistent with that for test functions.

\begin{theorem} \label{consistence} Let $f\in\mathcal{D}' _{L^{1}}(\mathbb{R}^{n})$. The ridgelet transform of $f$ is given by the function (\ref{ridgelet}), that is,
\begin{equation}
\label{eq consistence}
\langle \mathcal{R}_{\psi}{f}, \Phi\rangle= \int^{\infty
}_{0}\int^{\infty }_{-\infty }\int_{\mathbb{S}^{n-1}} \mathcal{R}_{\psi}f(\mathbf{u},b,a)\Phi(\mathbf{u},b,a) \frac{d\mathbf{u}dbda}{a^n}, \    \    \   \Phi\in\mathcal{S}(\mathbb{Y}^{n+1}).
\end{equation}
\end{theorem}
\begin{proof} By Schwartz' structural theorem \cite{scwartz}, we can write $f=\sum_{j=1}^{N}f_{j}^{(m_j)}$, where each $f_{j}\in L^{1}(\mathbb{R}^{n})$. 
Observe first that
$$
\langle f_{j}^{(m_{j})},\psi_{\mathbf{u},b,a}\rangle= \left(-a^{-1}\mathbf{u}\right) ^{m_{j}} \langle f_{j},(\psi^{(m_{j})})_{\mathbf{u},b,a}\rangle.
$$
On the other hand, since
$$
(-1)^{|m_{j}|}\frac{\partial^{|m|_{j}}}{\partial x^{m_{j}}} \mathcal{R}_{\overline{\psi}}^{t} \Phi= \mathcal{R}_{\overline{\psi}^{(m_{j})}} ^{t}\left(\left(-a^{-1}\mathbf{u}\right) ^{m_{j}} \Phi\right),
$$
the ridgelet transform $\mathcal{R}_{\psi}{f}$, defined via (\ref{def11}), satisfies $$\mathcal{R}_{\psi}(f_{j}^{(m_{j})})=\left(-a^{-1}\mathbf{u}\right) ^{m_{j}}\mathcal{R}_{\psi^{(m_{j})}}f_{j}.$$ Therefore, we may assume that $f\in L^{1}(\mathbb{R}^{n})$. But in the latter case, the result is a consequence of Proposition \ref{transpose proposition}.
\end{proof}
\begin{rem} Let us point out that (\ref{eq consistence}) holds in particular for compactly supported distributions $f\in\mathcal{E}'(\mathbb{R}^{n})$ or, more generally, for convolutors $f\in\mathcal{O}'_{C}(\mathbb{R}^{n})$. Furthermore, when $f\in\mathcal{O}'_{C}(\mathbb{R}^{n})$, one can easily check that $\mathcal{R}_{\psi}f\in C^{\infty}(\mathbb{Y}^{n+1})$.
\end{rem}
\section{On the Radon transform on $\mathcal{S}_{0}'(\mathbb{R}^{n})$}
\label{section Radon transform}

In this section we explain how one can define the Radon transform of Lizorkin distributions. Its connection with the ridgelet and wavelet transforms will be discussed in Section \ref{section desingularization}.

We begin with test functions. Helgason \cite{helgason} and Gelfand et al. \cite{gelfand5} gave the range theorem for the Radon transform on $\mathcal{S}(\mathbb{R}^{n})$. Indeed, its range $R(\mathcal{S}(\mathbb{R}^{n}))$ consists of the closed subspace of all those $\varrho\in \mathcal{S}(\mathbb{S}^{n-1}\times \mathbb{R})$ such that $\varrho$ is even on $\mathbb{S}^{n-1}\times \mathbb{R}$, i.e.,  $\varrho(-\mathbf{u},-p)=\varrho(\mathbf{u},p)$, and  $\int_{-\infty}^{\infty}p^{k}\varrho(\mathbf{u},p)$ is a $k$-th degree homogeneous polynomial in $\mathbf{u}$ for all $k\in\mathbb{N}_{0}$. The situation is not so satisfactory for the dual Radon transform $R^{\ast}$, because it does not map $\mathcal{S}(\mathbb{S}^{n-1}\times \mathbb{R})$ to $\mathcal{S}(\mathbb{R}^{n})$. Consequently, the duality relation (\ref{Radonduality}) fails to produce a definition for the Radon transform on $\mathcal{S}'(\mathbb{R}^{n})$. The Radon transform on $\mathcal{S}'(\mathbb{R}^{n})$ can be defined \cite{gelfand5,ludwig,ram}, but it does not take values in $\mathcal{S}'(\mathbb{S}^{n-1}\times \mathbb{R})$. The range $R(\mathcal{S}'(\mathbb{R}^{n}))$ is particularly complicated to describe in even dimensions $n$.

As Helgason points out \cite{helgason}, a more satisfactory situation is obtained if we restrict our attention to the smaller test function spaces $\mathcal{S}_{0}(\mathbb{R}^{n})$ and $\mathcal{S}_{0}(\mathbb{S}^{n-1}\times\mathbb{R})$. In such a case, 
\begin{equation} \label{radontest}R:\mathcal{S}_{0}(\mathbb{R}^{n})\to \mathcal{S}_{0}(\mathbb{S}^{n-1}\times\mathbb{R})
\end{equation}
and  
\begin{equation} \label{dualradontest}R^{\ast}:\mathcal{S}_{0}(\mathbb{S}^{n-1}\times\mathbb{R})\to\mathcal{S}_{0}(\mathbb{R}^{n}).
\end{equation}
We apply our results from Section \ref{ridgelet test functions} to deduce the following continuity result for $R$ and $R^{\ast}$.

\begin{co}\label{radonc1} The mappings (\ref{radontest}) and (\ref{dualradontest}) are continuous.
\end{co}
\begin{proof}  Let $\psi\in\mathcal{S}_{0}(\mathbb{R})$ have a reconstruction wavelet \cite{hols} $\eta\in\mathcal{S}_{0}(\mathbb{R})$, that is, one that satisfies
\begin{equation}
\label{waveletreconstruction}
c_{\psi,\eta}=\int^{\infty }_{0}\overline{\widehat{\psi }}(\omega){\widehat{\eta }}(\omega)\frac{d\omega}{\omega}=\int^{0}_{-\infty}\overline{\widehat{\psi }}(\omega){\widehat{\eta }}(\omega)\frac{d\omega}{|\omega|}\neq 0.
\end{equation}
     
From the one-dimensional reconstruction formula \cite{hols}, we obtain $c_{\psi,\eta}\mathrm{id}_{\mathcal{S}_{0}(\mathbb{S}^{n-1}\times \mathbb{R})}=\mathcal{M}_{\eta}\mathcal{W}_{\psi}$. By (\ref{rad-rid}), 
$R=c_{\psi,\eta}^{-1}(\mathcal{M}_{\eta}\mathcal{R}_{\psi})$,
 and so the continuity of $R$ follows from Theorem \ref{continuity theorem ridgelet} and Corollary \ref{waveletanalysisS}. Next, define the (continuous) multiplier operators 
\begin{equation}
\label{multiplier}J_{s}:\mathcal{S}(\mathbb{Y}^{n+1})\to \mathcal{S}(\mathbb{Y}^{n+1}), \ \ \ (J_{s}\Phi)(\mathbf{u},b,a)=a^{s}\Phi(\mathbf{u},b,a),  \ \  \ s\in\mathbb{R}.
\end{equation}
We have that
$$
R^{\ast}= \frac{1}{c_{\psi,\eta}} R^{\ast}\mathcal{M}_{\eta}J_{1-n}J_{n-1}\mathcal{W}_{\psi}= \frac{1}{c_{\psi,\eta}} \mathcal{R}^{t}_{\eta} J_{n-1} \mathcal{W}_{\psi}
$$
is continuous in view of Theorem \ref{continuity theorem inverse} and Corollary \ref{waveletanalysisS}.
\end{proof}

The mapping (\ref{dualradontest}) allows one to extend the definition of the Radon transform to $\mathcal{S}'_{0}(\mathbb{R}^{n})$.
\begin{de} \label{def4} The Radon transform
\begin{equation} \label{radondistributions}R:\mathcal{S}_{0}'(\mathbb{R}^{n})\to \mathcal{S}_{0}'(\mathbb{S}^{n-1}\times\mathbb{R})
\end{equation}
is defined via (\ref{Radonduality}).
\end{de}
Since  (\ref{radondistributions}) is the transpose of (\ref{dualradontest}), we obtain,
\begin{co}\label{radonc2} The Radon transform is continuous on $\mathcal{S}'_{0}(\mathbb{R}^{n})$.
\end{co}

Notice that the dual Radon transform (\ref{dualradontest}) is surjective \cite{helgason}. Therefore, the Radon transform is injective on $\mathcal{S}'_{0}(\mathbb{R}^{n})$. The restriction of (\ref{radondistributions}) to the subspaces $\mathcal{D}'_{L^{1}}(\mathbb{R}^{n})$, $\mathcal{E}'(\mathbb{R}^{n})$, $\mathcal{O}'_{C}(\mathbb{R}^{n})$, clearly coincides with the Radon transform treated by Hertle in \cite{hertle2}.
\section{ Ridgelet desingularization in $\mathcal{S}'_{0}(\mathbb{R}^{n})$}
\label{section desingularization}
The ridgelet transform of $f\in\mathcal{S}'_{0}(\mathbb{R}^{n})$ is in turn highly regular in ``the variables'' $b$ and $a$. This last section is devoted to prove this fact. We also give a ridgelet desingularization formula and establish the connection between the ridgelet, wavelet, and Radon transforms.

As mentioned in Subsection \ref{spaces}, we have $\mathcal{S}(\mathbb{Y}^{n+1})=\mathcal{D}(\mathbb{S}^{n-1})\hat{\otimes}\mathcal{S}(\mathbb{H})$. The nuclearity of the Schwartz spaces leads to the isomorphisms $\mathcal{S}'(\mathbb{Y}^{n+1})\cong \mathcal{S}'(\mathbb {H}, \mathcal {D}'(\mathbb S^{n-1}))\cong \mathcal{D}'(\mathbb{S}^{n-1},\mathcal{S}'(\mathbb{H}))$, the very last two spaces being spaces of vector-valued distributions \cite{silva,treves}. We shall identify these three spaces and write
\begin{equation}
\label{tensors}
\mathcal{S}'(\mathbb{Y}^{n+1})= \mathcal{S}'(\mathbb {H}, \mathcal {D}'(\mathbb S^{n-1}))=\mathcal{D}'(\mathbb{S}^{n-1},\mathcal{S}'(\mathbb{H})).
\end{equation}
The equality (\ref{tensors}) being realized via the standard identification 
\begin{equation}
\label{tensors2}
\left\langle F,\varphi\otimes \Psi \right\rangle=\left\langle\left\langle F,\Psi \right\rangle,\varphi\right\rangle=  \left\langle \left\langle F,\varphi \right\rangle,\Psi\right\rangle, \     \    \ \Psi\in\mathcal{S}(\mathbb{H}),\ \varphi\in\mathcal{D}(\mathbb{S}^{n-1}),
\end{equation}
Thus, given $F\in\mathcal{S}'(\mathbb{Y}^{n+1})$, the statement $F$ is smooth in $(b,a)$ has the clear interpretation $F\in C^{\infty}(\mathbb{H},\mathcal{D}'(\mathbb{S}^{n-1}))=\mathcal{D}'(\mathbb{S}^{n-1},C^{\infty}(\mathbb{H}))$. Moreover, we shall say that $F\in\mathcal{S}'(\mathbb{Y}^{n+1})$ is a function of slow growth in the variables $(b,a)\in\mathbb{H}$ if 
$
\left\langle F(\mathbf{u},b,a),\varphi(\mathbf{u})\right\rangle_{\mathbf{u}} 
$
is such for every $\varphi\in\mathcal{D}(\mathbb{S}^{n-1})$, namely, it is a function that satisfies the bound $$|\left\langle F(\mathbf{u},b,a),\varphi(\mathbf{u})\right\rangle_{\mathbf{u}} |\leq C\left(a^{s}+\frac{1}{a^{s}}\right)(1+|b|)^{s}, \  \  \  (b,a)\in\mathbb{H},$$
for some positive constants $C=C_{\varphi}$ and $s=s_{\varphi}$.

Notice also that $\mathcal{S}_{0}'(\mathbb{S}^{n-1}\times \mathbb{R})=\mathcal{S}_{0}'(\mathbb{R},\mathcal{D}'(\mathbb{S}^{n-1}))$ (again under the standard identification). This allows us to define the wavelet transform ($\psi\in\mathcal{S}_{0}(\mathbb{R})$),
$$\mathcal{W}_{\psi}:\mathcal{S}_{0}'(\mathbb{S}^{n-1}\times \mathbb{R})=\mathcal{S}_{0}'(\mathbb{R},\mathcal{D}'(\mathbb{S}^{n-1}))\to  \mathcal{S}'(\mathbb {H}, \mathcal {D}'(\mathbb S^{n-1}))=\mathcal{S}'(\mathbb{Y}^{n+1}),$$ 
by  direct application of the formula (\ref{wavelett2}) as a smooth vector-valued function $\mathcal{W}_{\psi}g:\mathbb{H}\to \mathcal{D}'(\mathbb{S}^{n-1})$, for $g\in\mathcal{S}_{0}'(\mathbb{S}^{n-1}\times\mathbb{R})$.  One can also check that this wavelet transform satisfies
\begin{equation}
\label{waveletvector}
\left\langle g,\mathcal{M}_{\overline{\psi}}\Phi\right\rangle=\int_{0}^{\infty}\int_{-\infty}^{\infty} \left\langle \mathcal{W}_{\psi}g(\mathbf{u},b,a),\Phi(\mathbf{u},b,a)\right\rangle_{\mathbf{u}}\frac{dbda}{a},
\end{equation}
for $g\in\mathcal{S}_{0}'(\mathbb{S}^{n-1}\times\mathbb{R})$ and $\Phi\in\mathcal{S}(\mathbb{Y}^{n+1}),$ where $\mathcal{M}_{\overline{\psi}}$ is as in (\ref{synthesisS}) (cf. \cite[Sect. 5 and 8]{PV} for comments on the vector-valued wavelet transform).

The relation between the Radon transform, the wavelet transform, and the ridgelet transform is stated in the following theorem, which also tells us that the ridgelet transform is regular in the location and scale parameters.

\begin{theorem} \label{theorem 1 ridgelet} Let $f \in \mathcal S'_{0}(\mathbb R^{n})$ and $\psi\in\mathcal{S}_{0}(\mathbb{R})$. Then,\begin{equation}
\label{eqdesingular}
\left\langle \mathcal{R}_{\psi}f,\Phi\right\rangle= \int_{0}^{\infty}\int_{-\infty}^{\infty} \left\langle \mathcal{W}_{\psi}(Rf)(\mathbf{u},b,a),\Phi(\mathbf{u},b,a)\right\rangle_{\mathbf{u}} \frac{dbda}{a^{n}}, \  \  \  \Phi\in\mathcal{S}(\mathbb{Y}^{n+1}). 
\end{equation} 
Furthermore, $\mathcal{R_{\psi}}f\in C^{\infty}(\mathbb H, \mathcal D'(\mathbb S^{n-1}))$ and it is of slow growth on $\mathbb{H}$. 
\end{theorem}
\begin{proof}
That $\mathcal{R}_{\psi}$ is smooth and of slow growth in the variables $b,a$  follows from (\ref{eqdesingular}) and the corresponding property for the wavelet transform. Let us show (\ref{eqdesingular}). The multiplier operator $J_{s}$ was introduced in (\ref{multiplier}). By (\ref{waveletvector}),
\begin{align*}
\int_{0}^{\infty}\int_{-\infty}^{\infty} \left\langle \mathcal{W}_{\psi}(Rf)(\mathbf{u},b,a),\Phi(\mathbf{u},b,a)\right\rangle_{\mathbf{u}} \frac{dbda}{a^{n}}&= \left\langle Rf, (\mathcal{M}_{\overline{\psi}}J_{1-n})\Phi \right\rangle
\\
&
=\left\langle f, (R^{\ast}\mathcal{M}_{\overline{\psi}}J_{1-n})\Phi\right\rangle
\\
&=
\langle f, \mathcal{R}^{t}_{\overline{\psi}}\Phi\rangle
\\
&
=
 \left\langle \mathcal{R}_{\psi} f, \Phi\right\rangle.
\end{align*}

\end{proof}
It should be emphasized that the relation (\ref{eqdesingular}) is consistent with the ridgelet transform of test functions, as follows from Theorem \ref{consistence} and (\ref{rad-rid}).

We end this article with a desingularization formula, a corollary of Theorem \ref{theorem 1 ridgelet}. The next result generalizes the extended Parseval's relation obtained in Proposition \ref{proposition Parseval}. 

\begin{co}[Ridgelet desingularization]\label{co desing} Let $f\in \mathcal{S}'_{0}(\mathbb{R}^{n})$ and let $\psi\in\mathcal{S}_{0}(\mathbb{R})$ be non-trivial. If $\eta\in\mathcal{S}_{0}(\mathbb{R})$ is a reconstruction neuronal activation function for $\psi$, then
\begin{equation}
\label{eq desingular 2}
\langle f,\phi\rangle= \frac{1}{K_{\psi,\eta}}\int_{0}^{\infty}\int_{\mathbb{R}} \left\langle \mathcal{W}_{\psi}(Rf)(\mathbf{u},b,a),\mathcal{R}_{\overline{\eta}}\:\phi (\mathbf{u},b,a)\right\rangle_{\mathbf{u}} \frac{dbda}{a^{n}},
\end{equation}
for all $\phi\in\mathcal{S}_{0}(\mathbb{R}^{n})$.
\end{co}
\begin{proof} By Theorem \ref{reconstruction distributions},
$$
\left\langle f,\phi\right\rangle= \frac{1}{K_{\psi,\eta}} \left\langle f,\mathcal{R}^{t}_{\overline{\psi}}\mathcal{R}_{\overline{\eta}}\:\phi\right\rangle=\frac{1}{K_{\psi,\eta}}\left\langle \mathcal{R}_{\psi}f,\mathcal{R}_{\overline{\eta}}\: \phi\right\rangle.
$$
The desingularization formula (\ref{eq desingular 2}) follows then from (\ref{eqdesingular}).
\end{proof}

\end{document}